\documentclass[11pt]{article}

\RequirePackage[OT1]{fontenc}
\RequirePackage{amsthm,amsmath, amssymb, enumerate}
\RequirePackage[numbers]{natbib}
\RequirePackage[colorlinks,citecolor=blue,urlcolor=blue]{hyperref}

 \numberwithin{equation}{section} 
\theoremstyle{plain}
\newtheorem{theorem}{Theorem}[section]
\newtheorem{lemma}{Lemma}[section]
\newtheorem{definition}{Definition}[section]
\newtheorem{example}{Example}[section]
\newtheorem{corollary}{Corollary}[section]
\newtheorem{remark}{Remark}[section]

\newcommand{\R}{\mathbb R}

\begin{document}

\title{Stein's density approach  and information inequalities}


\author{{Christophe} {Ley} and {Yvik} {Swan}}




\maketitle 
\begin{abstract} We provide a new perspective on Stein's so-called
  density approach by introducing a new operator and characterizing
  class which are valid for a much wider family of probability
  distributions on the real line.  We prove an elementary
  factorization property of this operator and propose a new Stein
  identity which we use to derive information inequalities in terms of
  what we call the \emph{generalized Fisher information distance}. We
  provide explicit bounds on the constants appearing in these
  inequalities for several important cases.   We conclude with a comparison
  between our results and known results in the Gaussian case, hereby
  improving on several known inequalities from the literature.
\end{abstract}


\section{Introduction}
Charles Stein's crafty exploitation of the characterization 
\begin{equation}
  \label{eq:2}
  X \sim \mathcal{N}(0, 1) \Longleftrightarrow \mathrm{E} \left[
    f'(X)-Xf(X) \right]=0 \, \mbox{ for all bounded } f \in C^1(\R)
\end{equation}
has given birth to a ``method'' which is now an acclaimed tool both in
applied and in theoretical probability. The secret of the ``method''
lies in the structure of the operator $\mathcal{T}_{\phi}f(x) :=
f'(x) - xf(x)$ and in the flexibility in the choice of test functions
$f$.  For the origins  we refer the reader to
\cite{S81,St86,S72}; for an overview of the more recent achievements
in this field we refer to the monographs
\cite{NP11,BC05,ChGoSh11} or the review articles
\cite{NP10,Ro11}.


Among the many ramifications and extensions that the method has known,
so far the connection with information theory has gone relatively
unexplored.  Indeed while it has
long been known that Stein identities such as  \eqref{eq:2} are
related to information theoretic tools and concepts (see, e.g.,
\cite{Jo04,  KoHaJo05, CoTh06}), to the best of our knowledge the only
references to explore this connection upfront are \cite{BaJoKoMa10} in
the context of compound Poisson approximation,  and more recently 
\cite{Sa12,sason2012entropy} for Poisson and Bernoulli
approximation.  In this paper and the
companion paper \cite{LS11c} we extend Stein's characterization of the
Gaussian \eqref{eq:2} to a broad class of univariate distributions
and, in doing so, provide an adequate framework in which the
connection with information distances becomes transparent. 

 
The structure of the present paper is as follows. In Section
\ref{sub:st_inf} we provide the new perspective on the density
approach from \cite{StDiHoRe04} which allows to extend this
construction to virtually any absolutely continuous probability
distribution on the real line. In Section \ref{sec:conn-with-inform}
we exploit the structure of our new operator to derive a family of
Stein identities through which the connection with information distances
becomes evident. In Section \ref{newsection} we compute bounds on the constants
appearing in our inequalities; our method of proof is, to the best of
our knowledge, original. Finally in Section
\ref{sec:applications} we discuss specific examples. 
 
\section{The density approach} \label{sub:st_inf}

Let $\mathcal G$ be the collection of positive real functions $x \mapsto p(x)$ such
that (i) their support $S_p := \left\{ x \in \R \, : \; p(x) \mbox{
    (exists and) is positive}\right\}$ is an interval with
closure $\bar S_p= [a, b]$, for some $-\infty\le a < b \le \infty$, 
(ii) they are differentiable (in the usual sense) at every point in
$(a, b)$ with derivative $x \mapsto p'(x) := \frac{d}{dy}p(y) |_{y=x}$ and (iii)
 $\int_{S_p}p(y)dy =1$. Obviously, each $p\in \mathcal{G}$ is the
 density (with respect to
 the Lebesgue measure) of an absolutely continuous random
 variable. Throughout  we adopt the  convention 
 \begin{equation*}
   \label{eq:1}
   \frac{1}{p(x)} =  \left\{
     \begin{array}{cl}
       \frac{1}{p(x)} & \mbox{if } x \in S_p\\
0 & \mbox{otherwise};
     \end{array}
\right.
 \end{equation*}
this implies, in
particular, that    $p(x) /
p(x) = \mathbb{I}_{S_p}(x)$, the indicator function of the support
$S_p$. As final notation, for  $p \in \mathcal{G}$ we  write
${\rm E}_p[l(X)] := \int_{S_p} l(x) p(x) dx.$ 


With this setup in hand we are ready to provide the two main definitions of
this paper (namely, a class of functions and an operator) and to state
and prove our first main result (namely, a characterization).

\begin{definition} \label{def:class} To $p \in \mathcal G$ we
  associate (i) the collection $\mathcal{F}(p)$ of functions $f: \R\to
  \R$ such that the mapping $x \mapsto f(x)p(x)$ is  differentiable on
  the interior of $S_p$ and $f(a^{+})p(a^{+}) = f(b^{-})p(b^{-}) =0$, 
  and (ii) the operator
  $\mathcal{T}_p : \mathcal{F}(p) \to \R^\star : f \mapsto
  \mathcal{T}_pf$ defined through
\begin{equation}\label{eq:st_op}  \mathcal{T}_pf:
  \R\rightarrow\R:x  \mapsto \mathcal{T}_pf(x) :=
  \frac{1}{p(x)} \left. \frac{d}{dy}(f(y)p(y))\right|_{y=x}.\end{equation} 
We call $\mathcal{F}(p)$ the class of \emph{test functions} associated
with $p$, and $\mathcal{T}_p$ the \emph{Stein operator} associated with
$p$. 
\end{definition}

\begin{theorem}\label{theo1} Let $p, q \in \mathcal G$ and let $Q(b) =
  \int_a^bq(u)du$.  Then $\int_{-\infty}^{+\infty} \mathcal{T}_pf(y)
  q(y) dy=0$ for all $f \in \mathcal{F}(p)$ if, and only if,
  ${q(x)}=p(x){Q(b)}$ for all $x \in S_p$. 
\end{theorem}
 
\begin{proof} If $Q(b)=0$ the statement holds trivially.  We now take
  $Q(b)>0$. To see the sufficiency, note that the
  hypotheses on $f$, $p$ and $q$ guarantee that 
  \begin{align*}
    \int_{-\infty}^{\infty} \mathcal{T}_pf(y) q(y) dy & =   Q(b) \int_a^b 
    \frac{d}{du}(f(u)p(u)) |_{u=y}dy  \\ 
& 
= Q(b) \left(  f(b^{-})p(b^{-}) - f(a^{+})p(a^{+})  \right) = 0.
  \end{align*}
To see the necessity, first note that the condition   $\int_{\R}
\mathcal{T}_pf(y) q(y) dy=0$ implies that the function $y \mapsto
\mathcal{T}_pf(y) q(y)$ be Lebesgue-integrable
. Next define for $z \in
\R$ the
function  
\begin{equation*}
  l_z(u):= ({\mathbb{I}}_{(a, z]}(u) -
P(z))\mathbb{I}_{S_p}(u)
\end{equation*}
 with $P(z):=\int_{a}^z p(u)du$, which
satisfies
\begin{equation*}
  \int_a^bl_z(u)p(u)du = 0.
\end{equation*}
Then the function 
$$
    f_z^p(x) : = \frac{1}{p(x)} \int_{a}^x l_z(u) p(u) du \left( =
      -\frac{1}{p(x)} \int_{x}^b l_z(u) p(u) du \right) 
$$
  belongs to $\mathcal{F}(p)$ for all $z$ and satisfies the
equation
$$\mathcal{T}_pf_z^p(x) =
  l_z(x)$$
for all $x \in S_p$. 
For this choice of test function we then obtain 
\begin{align*}
\int_{-\infty}^{+\infty}\mathcal{T}_pf_z^p(y) q(y)dy & =    \int_{-\infty}^{+\infty}l_{z}(y)
  q(y)dy= (Q(z) - P(z)Q(b)) \mathbb I_{S_p}(z),
\end{align*}
with $Q(z):=\int_{a}^z q(u)du$. Since this integral equals zero by hypothesis, it follows that $  {Q(z)} = P(z){Q(b)}$ for all $z \in S_p$, hence the claim holds. 
 \end{proof}
The above is, in a sense, nothing more than a peculiar
statement of what is often referred to as a   ``Stein
characterization''.  Within the more 
conventional framework of real random variables having absolutely 
continuous densities, Theorem \ref{theo1} reads as follows. 
\begin{corollary}[The density approach]\label{cor:dens-approach} Let
  $X$ be an absolutely continuous random variable with density $p \in
  \mathcal{G}$. Let $Y$ be another absolutely continuous random
  variable. Then 
${\rm E} \left[ \mathcal{T}_pf(Y) \right]=0$ for all $f \in \mathcal{F}(p)$
if, and only if,   either ${\rm P}(Y\in S_p)=0$ or ${\rm P}(Y\in S_p)>0$ and
\begin{equation*}
  {\rm P}\left(  \left. Y \le z \, \right| Y \in S_p \right) = {\rm P}(X \le z)
\end{equation*}
for all $ z \in S_p $.
\end{corollary}
 Corollary \ref{cor:dens-approach} extends the density
approach from \cite{StDiHoRe04} or  \cite{CS11,
  ChGoSh11} to a much wider class of distributions; it also contains 
 the Stein characterizations for the Pearson given in \cite{S01} and
 the   more recent general 
characterizations  studied  in \cite{Do12, GoRe12}.  There is,
however, a significant  shift  operated between our ``derivative of a
product'' operator \eqref{eq:st_op} and the standard way of writing
these operators in the literature.  Indeed, while 
  one can
always distribute the derivative in 
\eqref{eq:st_op}  to obtain (at least formally)  the expansion
\begin{equation}
  \label{eq:3}
\mathcal{T}_pf(x) = \left( f'(x) + \frac{p'(x)}{p(x)}f(x) \right)
\mathbb I_{S_p}(x),
\end{equation}
the latter requires $f$ 
be differentiable on $S_p$ in order to make sense. We do not require
this, neither do we
require that each summand in \eqref{eq:3} be well-defined on
$S_p$ nor do we need to impose integrability conditions on $f$ for
Theorem~\ref{theo1} (and thus Corollary \ref{cor:dens-approach}) to hold!
Rather, our definition of $\mathcal{F}(p)$ allows to
 identify a collection of minimal conditions on the class of
 test functions $f$ for the resulting operator $\mathcal{T}_p$ to
 be orthogonal to $p$ w.r.t. the Lebesgue measure, and thus
 characterize $p$.

\begin{example}
  Take $p=\phi$,  the standard Gaussian. Then $\mathcal{F}(\phi)$ is
composed of all real-valued  functions $f$ such that (i) $x \mapsto f(x)
e^{-x^2/2}$ is differentiable on $\R$ and  (ii) $\lim_{x\to \pm
  \infty}f(x)  e^{-x^2/2}= 0$. In particular  $\mathcal{F}(\phi)$
  contains the collection  of all differentiable bounded functions and  
$$\mathcal T_{\phi}f (x) = f'(x) - xf(x),$$ 
which is  Stein's well-known operator for characterizing the
Gaussian (see, e.g., \cite{S72, BC05,ChGoSh11}). There are of course
many other subclasses that can be of interest. For example the class
$\mathcal{F}(\phi)$  also contains the collection 
of functions $f(x)= -f_0'(x)$ with $f_0$ a twice  differentiable bounded
function; for these we get 
$$\mathcal T_{\phi}f (x) = xf_0'(x)-f_0''(x),$$ 
the generator  of an Ornstein-Uhlenbeck process, 
see \cite{Ba90,G91, NP11}. The class 
 $\mathcal{F}(\phi)$ as well contains the
collection of functions of the form $f(x)=H_n(x) f_0(x) $ for $H_n$
the $n$-th Hermite polynomial and  $f_0$ any differentiable and
bounded function. For these $f$ we get 
\begin{equation*}
  \mathcal T_{\phi}f (x) = H_n(x) f_0'(x) -H_{n+1}(x) f_0(x), 
\end{equation*}
an operator already discussed in \cite{GR05} (equation (38)).  
\end{example}

\begin{example}
  Take $p = Exp $ 
 the standard
rate-one exponential
distribution. Then $\mathcal{F}(Exp)$ is composed of all real-valued
  functions $f$ such that (i) $x \mapsto f(x) e^{-x}$ is
differentiable on $(0, +\infty)$, (ii)   $f(0) = 0$ and (iii) $\lim_{x \to
  +\infty} f(x)e^{-x} = 0$. In particular $\mathcal{F}(Exp)$ contains the
 collection  of all differentiable bounded functions such that
 $f(0)=0$ and 
 \begin{equation*}
   \mathcal T_{Exp}f(x) = \left(f'(x) -  f(x)\right)\mathbb{I}_{[0,\infty)}(x),
 \end{equation*}
 the operator usually associated to the exponential, see \cite{Lu94,
   Pi04, StDiHoRe04}. The class
 $\mathcal{F}(Exp)$ also contains the collection of functions of the
 form $f(x) = x f_0(x)$ for $f_0$ any differentiable bounded function. For
 these $f$ we get 
 \begin{equation*}
   \mathcal T_{Exp}f(x) = \left(xf_0'(x) + (1-x)  f_0(x)\right)\mathbb{I}_{[0,\infty)}(x),
 \end{equation*}
an operator put to use in \cite{ChFuRo11}.
\end{example}

\begin{example}
  Finally take $p = Beta(\alpha, \beta)$
the beta distribution with parameters
$(\alpha, \beta) \in \R_0^{+}\times \R_0^{+}$. Then
$\mathcal{F}(Beta(\alpha, \beta))$ is composed of all real-valued
functions $f$ such that (i)  $x \mapsto f(x) x^{\alpha-1}(1-x)^{\beta-1} $ is
differentiable on $(0, 1)$, (ii) $ \lim_{x \to 0}  f(x)
x^{\alpha-1}(1-x)^{\beta-1} = 0$ and (iii) $ \lim_{x \to 1}  f(x)
x^{\alpha-1}(1-x)^{\beta-1} = 0$. In particular
$\mathcal{F}(Beta(\alpha, \beta))$ contains the
 collection  of functions of the form
$   f(x) = (x(1-x))f_0(x)$
with $f_0$ any differentiable bounded function. For
these $f$ we get 
 \begin{equation*}
   \mathcal T_{Beta(\alpha,\beta)}f(x) =  \left( \left( \alpha(1-x) -
       \beta x \right) f_0(x) + x(1-x)f_0'(x) \right) \mathbb{I}_{[0,1]}(x),
 \end{equation*}
an operator recently put to use in, e.g., \cite{GoRe12,Do12}. 
\end{example}

There are obviously many more distributions that can be tackled as in
the previous examples (including the Pearson case from 
\cite{S01}), which we leave  to the interested reader. 



\section{Stein-type identities and the generalized Fisher information distance}
\label{sec:conn-with-inform}

 It has long been known that, in certain favorable circumstances, the
properties of the Fisher information or of the Shannon entropy can be
used quite effectively to prove  information
theoretic   central limit theorems; the 
early references in this vein are \cite{Sh75, Br82, 
  BA86,LI59}. 
Convergence in information CLTs  is
generally  studied  in terms of information (pseudo-)distances 
such as the \emph{Kullback-Leibler divergence}  between two
densities  $p$ and $q$, defined as
 \begin{equation}\label{eq:26} 
   d_{\rm KL}(p|| q) = {\rm E}_q \left[ \log \left( \frac{q(X)}{p(X)} \right) \right],
 \end{equation}
or the \emph{Fisher information distance}
\begin{equation}\label{eq:11}
  \mathcal{J}(\phi,q)={\rm E}_q\left[\left(X+\frac{q'(X)}{q(X)}\right)^2\right]
\end{equation}
 which measures deviation
between any density $q$ and the standard
Gaussian $\phi$. Though they allow for extremely elegant proofs,
convergence in the sense of \eqref{eq:26} or \eqref{eq:11} results in
very  strong statements. Indeed both \eqref{eq:26} and \eqref{eq:11}  are known  to dominate
more ``traditional'' probability metrics. More precisely we have,
on the one hand,  \emph{Pinsker's inequality} 
\begin{equation}\label{eq:13}
  d_{\rm TV} (p, q) \le \frac{1}{\sqrt 2} \sqrt{    d_{\rm KL}(p|| q)},
\end{equation}
for $d_{\rm TV}(p,q)$ the \emph{total variation distance} between the laws $p$ and
$q$ (see, e.g., \cite[p. 429]{GS02}), and, on the other hand,  
\begin{equation}\label{eq:25}
  d_{L^1}(\phi, q) \le \sqrt 2 \sqrt{\mathcal{J}(\phi, q)}
\end{equation}
for $d_{L^1}(\phi, q)$ the \emph{$L^1$ distance} between the laws $\phi$
and $q$ (see \cite[Lemma 1.6]{MR2128239}). These
information inequalities  show that convergence in the sense of
 \eqref{eq:26} or \eqref{eq:11} implies  convergence in total 
 variation or  in $L^1$, for example. Note that one can
 further use De Brujn's identity on \eqref{eq:13} to deduce that
 convergence in Fisher information is itself stronger than convergence in
 relative entropy.

 While Pinsker's inequality \eqref{eq:13} is valid irrespective of the
 choice of $p$ and $q$ (and enjoys an extension to discrete random
 variables), both \eqref{eq:11} and \eqref{eq:25} are reserved 
 for Gaussian convergence.  Now there exist extensions of the distance 
 \eqref{eq:11} to non-Gaussian distributions (see \cite{BaJoKoMa10}
 for the discrete case) which, as could be expected, have also been
 shown to dominate the more traditional probability metrics. There is,
 however,  no general counterpart of Pinsker's inequality for
 the Fisher information distance \eqref{eq:11}; at least there exists,
 to the best of our knowledge, no inequality in the literature which
extends \eqref{eq:25} to a general couple of densities $p$ and $q$. 

In this section we use the density approach outlined in Section
\ref{sub:st_inf} to construct Stein-type identities which
provide the required   extension of \eqref{eq:25}.  More precisely, we will show that
a wide family of probability metrics (including the \emph{Kolmogorov}, the
\emph{Wasserstein} and the $L^1$ distances) is dominated by the quantity 
\begin{equation}
  \label{eq:19}
  \mathcal{J}(p,q):={\rm E}_q\left[\left(\frac{p'(X)}{p(X)}-\frac{q'(X)}{q(X)}\right)^2\right].
\end{equation}
Our bounds,
moreover,  contain an explicit constant which will be shown in Section
\ref{newsection} to be at worst as good as  the best   bounds in all
known instances. In the
spirit of \cite{BaJoKoMa10} we call \eqref{eq:19} the
\emph{generalized Fisher information distance} between the densities
$p$ and $q$, although here we slightly abuse of language since
\eqref{eq:19} rather defines a {pseudo-distance} than a \emph{bona
  fide} metric between probability density functions.

We start with an elementary statement which relates, for $p
\neq q$, the Stein operators $\mathcal{T}_p$ and $\mathcal{T}_q$
through the difference of their respective \emph{score functions} $\frac{p'}{p}$ and $\frac{q'}{q}$. 
 
\begin{lemma}  \label{lemma:facto} Let $p$ and $q$ be  probability
  density functions in $\mathcal G$ with respective supports $S_p$ and $S_q$. Let $S_q\subseteq S_p$ and define 
$$
  r(p, q)(x) :=    \left(\frac{p'(x)}{p(x)} - \frac{q'(x)}{q(x)}\right)\mathbb{I}_{S_p}(x).
$$
Suppose that
  $\mathcal{F}(p) \cap \mathcal{F}(q) \neq \emptyset$.  Then, for all $f \in \mathcal{F}(p)\cap \mathcal{F}(q)$, we have
$$
    \mathcal{T}_pf(x) = \mathcal{T}_qf(x)+ f (x) r(p, q)(x)+    \mathcal{T}_pf(x)\mathbb{I}_{S_p\setminus S_q}(x),
$$
and therefore 
\begin{equation}
  \label{eq:9}
  {\rm E}_q \left[ \mathcal{T}_pf(X) \right] = {\rm E}_q \left[ f(X) r(p, q)(X)
  \right].
\end{equation}
\end{lemma}\vspace{2mm}

\begin{proof}
Splitting $S_p$ into $S_q\cup \{S_p\setminus S_q\}$, we have 
$$f(y)p(y) = f(y)q(y)p(y)/q(y)\mathbb{I}_{S_q}(y)+ f(y)p(y)\mathbb{I}_{S_p\setminus S_q}(y)$$  for any real-valued function $f$.
At any $x$ in the interior of $S_p$ we thus can write
 \begin{align*} & \mathcal{T}_pf(x)\\  
 						      &  = \frac{\left.\frac{d}{dy}(f(y)q(y) p(y)/q(y))\right|_{y=x}}{p(x)} \mathbb{I}_{S_q}(x)+\mathcal{T}_pf(x)\mathbb{I}_{S_p\setminus S_q}(x)  \\
 						      & =
                                                      \frac{\left.\frac{d}{dy}(f(y)q(y))\right|_{y=x}}{p(x)}
                                                      \frac{p(x)}{q(x)}+
                                                      f(x)q(x)\frac{\left.\frac{d}{dy}(
                                                          p(y)/q(y))\right|_{y=x}}{p(x)}+\mathcal{T}_pf(x)\mathbb{I}_{S_p\setminus S_q}(x)  \\ 
						      & =
                                                      \mathcal{T}_qf(x) +
                                                      f(x)\frac{q(x)}{p(x)}\left.\frac{d}{dy}
                                                        ( p(y)/q(y))\right|_{y=x}+\mathcal{T}_pf(x)\mathbb{I}_{S_p\setminus S_q}(x)  .
						      \end{align*}
The first claim readily follows by simplification, the second by taking
expectations under $q$ which cancels  the first term
$\mathcal{T}_qf(x)$ (by definition) as well as  the third term
$\mathcal{T}_pf(x)\mathbb{I}_{S_p\setminus S_q}(x)$ (since the
supports do not coincide).  
\end{proof}
\begin{remark}
  Our proof of Lemma \ref{lemma:facto} may seem circumvoluted; indeed
  a much easier proof is obtainable by writing $\mathcal{T}_p$ under the
  form \eqref{eq:3}. We nevertheless stick to the ``derivative of a
  product'' structure of our operator because this dispenses us with
  superfluous -- and, in some cases, unwanted -- differentiability
  conditions on the test functions.
\end{remark}

From   identity \eqref{eq:9} we deduce the following immediate
result, which requires no proof.
\begin{lemma}
  \label{lem:stein-type-ident}
Let $p$ and $q$ be  probability   density functions in $\mathcal G$ with respective supports $S_q\subseteq S_p$. Let $l$
be a real-valued function  such 
 that   ${\rm 
  E}_p[ l(X)]$ and ${\rm E}_q[ l(X)]$ exist; also suppose that  there
exists $f \in \mathcal{F}(p)\cap \mathcal{F}(q)$ such that 
\begin{equation}\label{eq:steineq22}\mathcal{T}_pf(x) =   (l(x)-
  {\rm E}_p[l(X)])\mathbb{I}_{S_p}(x);\end{equation} 
we denote this function $f_l^p$.
Then 
\begin{equation} \label{eq:fundeq}  {\rm E}_q[l(X)] -{\rm E}_p[l(X)] =  
{\rm E}_q [{f}_l^p(X) r(p , q)(X)].
\end{equation}

\end{lemma}

The identity \eqref{eq:fundeq} belongs to the family of so-called
``Stein-type identities'' discussed for instance in
\cite{GR05,CP89,APP11}.  In order to be of use, such identities need
to be valid over a large class of test functions $l$. Now it is
immediate to write out the solution $f_l^p$ of the so-called ``Stein
equation'' \eqref{eq:steineq22} explicitly for any given $p$ and $l$;
it is therefore relatively simple to identify under which conditions
on $l$ and $q$ the requirement $f_l^p \in \mathcal{F}(q)$ is verified
(since $f_l^p \in \mathcal{F}(p)$ is anyway true).
\begin{remark}
  For instance, for $p=\phi$ the standard Gaussian, one easily sees
  that $\lim_{x\rightarrow\pm\infty}f_l^\phi(x)=0$, hence, when
  $S_q=S_\phi=\R$, $q$ only has to be (differentiable and) bounded for
  $f_l^\phi$ to belong to $\mathcal{F}(q)$.  However, when $S_q\subset
  \R$, then $q$ has to satisfy, moreover, the stronger condition of
  vanishing at the endpoints of its support $S_q$ since $f_l^\phi$
  needs not equal zero on any finite points in $\R$.
  \end{remark}
%
%
We shall see in the next section that the required conditions for
$f_l^p\in\mathcal{F}(q)$ are satisfied in many important cases by wide
classes of functions $l$. The resulting flexibility makes
\eqref{eq:fundeq} a surprisingly powerful identity, as can be seen
from our next result. 

\begin{theorem}\label{cor:gener-bound-dist}
Let $p$ and $q$ be  probability   density functions in $\mathcal G$
with respective supports $S_q\subseteq S_p$ and such that 
  $\mathcal{F}(p) \cap \mathcal{F}(q) \neq \emptyset$.
Let 
\begin{equation}\label{eq:hdizst}d_{\mathcal H} (p, q) = \sup_{l \in
    \mathcal{H}} \left|{\rm E}_q[l(X)] - {\rm E}_p[l(X)]
  \right|\end{equation} 
for some class of functions $\mathcal{H}$. Suppose that for all $l
\in \mathcal{H}$ the function $f_l^p$, as
defined in \eqref{eq:steineq22}, exists and satisfies $f_l^p \in
\mathcal{F}(p) \cap \mathcal{F}(q)$. Then 
\begin{equation}
  \label{eq:4}
  d_{\mathcal H} (p, q) \le \kappa_{\mathcal{H}}^p \sqrt{\mathcal{J}(p, q)},
\end{equation}
where 
\begin{equation}
  \label{eq:5}
  \kappa_{\mathcal{H}}^p =  \sup_{l \in \mathcal{H}}\sqrt{{\rm E}_q [(f_l^p(X))^2]}
\end{equation}
and 
\begin{equation}
  \label{eq:6}
  \mathcal{J}(p, q) = {{\rm E}_q[(r(p,q)(X))^2]},
\end{equation}
the generalized Fisher information distance between
the densities $p$ and $q$.   
\end{theorem}

This theorem implies that all probability metrics that can be written in
the form \eqref{eq:hdizst} are bounded by the generalized Fisher
information distance $\mathcal{J}(p,q)$ (which, of course, can be
infinite for certain choices of $p$ and $q$). 
%
Equation \eqref{eq:4} thus represents the announced extension of \eqref{eq:25} to any couple of densities $(p,q)$ and hence constitutes, in a sense, a counterpart to
Pinsker's inequality \eqref{eq:13} for the Fisher information
distance. We will see in Section \ref{sec:applications} how this
inequality reads for specific choices of $\mathcal{H}$, $p$ and $q$.

\section{Bounding the constants}\label{newsection}


  The constants $\kappa_{\mathcal{H}}^p$ in \eqref{eq:5} depend on
  both densities $p$ and $q$ and therefore, to be fair, should be
  denoted $\kappa_{\mathcal{H}}^{p,q}$.  Our notation is nevertheless
  justified because 
  we always have
\begin{equation}\label{eq:10}
  \kappa_{\mathcal{H}}^{p}\le \sup_{l \in
    \mathcal{H}}\|f_l^p\|_{\infty},
\end{equation}
where the latter bounds (sometimes referred to as \emph{Stein factors} or
\emph{magic factors}) do not depend on $q$ and have been computed for 
many choices of $\mathcal{H}$ and $p$. Consequently,
$\kappa_{\mathcal{H}}^p$ is finite in many known cases
-- including, of course, that of a  Gaussian target.

\begin{example} \label{ex:bounding-constants} Take $p=\phi$, the
  standard Gaussian. Then, from \eqref{eq:10}, we get the bounds (i)
  $\kappa_{\mathcal{H}}^p \le \sqrt{{\pi}/{2}}$ for $\mathcal{H}$ the
  collection of Borel functions in $[0,1]$ (see \cite[Theorem
  3.3.1]{NP11}); (ii) $\kappa_{\mathcal{H}}^p \le \sqrt{2{\pi}}/{4}$
  for $\mathcal{H}$ the class of indicator functions for lower
  half-lines (see \cite[Theorem 3.4.2]{NP11}); and (iii)
  $\kappa_{\mathcal{H}}^p \le \sqrt{\pi/2} \sup_{l\in \mathcal{H}}
  \min \left( \|l-{\rm E}_p \left[ l(X) \right]\|_{\infty}, 2 \|
    l'\|_{\infty} \right)$ for $\mathcal{H}$ the class of absolutely
  continuous functions on $\R$ (see \cite[Lemma 2.3]{ChSh05}). See
  also \cite{BC05,ChGoSh11,NP11, RO12} for more examples.
  \end{example}

Bounds such as \eqref{eq:10}
are sometimes too rough to be satisfactory. We now provide an
alternative bound for $\kappa_{\mathcal{H}}^p$ which, remarkably,
improves upon the best known bounds even in well-trodden cases such as the
Gaussian. 
 We focus on target
densities of the form
\begin{equation} 
  \label{eq:15}
  p(x) = c e^{-d |x|^{\alpha}} \mathbb{I}_S(x),\quad \alpha\geq1,
\end{equation}
with $S$ a scale-invariant subset of $\R$ (that is, either $\R$ or the
open/closed positive/negative real half lines), $d>0$ some constant
and $c$ the appropriate normalizing constant. 
The exponential, the Gaussian  or the
limit distribution for the Ising model on the complete graph from
\cite{CS11}  are all of the form \eqref{eq:15}. Of course, for $S=\R$,
\eqref{eq:15} represents power exponential densities.




\begin{theorem}\label{prop:bound} Take  $p\in\mathcal{G}$ as in \eqref{eq:15} and $q\in\mathcal{G}$ such that $S_q= S$. Consider $h : \R \to \R$ some Borel   function with $p$-mean ${\rm E}_p \left[ h(X) \right] =
  0$. Let    $f_h^p$ be the unique bounded solution
  of the Stein equation 
  \begin{equation}
    \label{eq:18}
    \mathcal{T}_pf(x) = h(x).
  \end{equation}
Then 
  \begin{equation}
    \label{eq:17}
\sqrt{{\rm E}_q \left[ \left( f_h^p(X) \right)^2 \right]}\le \frac{||h||_\infty}{2^{\frac{1}{\alpha}}}.
  \end{equation} 
\end{theorem}
 
\begin{proof}

Under the assumption that ${\rm E}_p[h(X)]=0$, the unique bounded
solution of \eqref{eq:18} is  given by   
\begin{equation*}
  f_h^p(x) =       \left\{
    \begin{array}{cc}
 \dfrac{1}{p(x)}\displaystyle{\int_{-\infty}^x h(y) p(y) dy} & \mbox{ if } x \le 0,\\
 \dfrac{-1}{p(x)}\displaystyle{\int_x^\infty h(y) p(y) dy} & \mbox{ if } x \ge 0,
    \end{array}
\right.
\end{equation*}
the function being, of course,  put to 0 if $x$ is outside the
support of $p$. 
Then
\begin{align*}
  {\rm E}_q \left[ (f_h^p(X))^2 \right] & = \int_{-\infty}^{0} q(x) 
\left(\frac{1}{p(x)} \int_{-\infty}^x h(y) p(y) dy \right)^2dx \\
& \quad \quad + \int_{0}^{\infty} q(x) 
\left(\frac{1}{p(x)} \int_x^\infty h(y) p(y) dy \right)^2dx \\
& =: I^-+I^+,
\end{align*}
where    $I^- = 0$ (resp., $I^+=0$)  if $\bar{S}=\R^+$  (resp., 
$\bar{S}=\R^-$). 

We first tackle $I^-$. Setting $p(x)=c e^{-d |x|^{\alpha}} \mathbb{I}_S(x)$ and using Jensen's inequality, we get
\begin{align*} I^- & =  \int_{-\infty}^0 q(x)\left( e^{d|x|^\alpha}\int_{-\infty}^x  h(u)e^{-d|u|^\alpha} du\right)^2 dx \\
&\leq \int_{-\infty}^0 q(x)\left(e^{d|x|^\alpha}\int_{-\infty}^x  |h(u)|e^{-d|u|^\alpha} du\right)^2 dx \\
					& \le   \int_{-\infty}^0 q(x)  \left(e^{2d|x|^\alpha} \int_{-\infty}^x  h^2(u)e^{-2d|u|^\alpha} du\right) dx\\
					&=\frac{1}{2^{1/\alpha}} \int_{-\infty}^0 q(x)  \left(e^{2d|x|^\alpha} \int_{-\infty}^{2^{1/\alpha}x}  h^2(u/2^{1/\alpha})e^{-d|u|^\alpha} du\right) dx,\end{align*}
where the last equality follows from a simple change of variables. Applying  H\"older's inequality we obtain
\begin{align*}  I^- & \leq \frac{\gamma_q^{1/2}}{2^{1/\alpha}}\sqrt{ \int_{-\infty}^0 q(x)  \left(e^{2d|x|^\alpha} \int_{-\infty}^{2^{1/\alpha}x}  h^2(u/2^{1/\alpha})e^{-d|u|^\alpha} du\right)^2 dx}=:I^-_1, \end{align*}
where $\gamma_q={\rm P}_q(X<0) := \int_{-\infty}^0q(x)dx$. Repeating
the  \emph{Jensen's inequality-change of variables-H\"older's
  inequality} scheme once more yields 
$$I^-\le I_1^-\leq I^-_2$$
with 
\begin{align*} 	I^-_2 =	  \frac{\gamma_q^{\frac{1}{2}+\frac{1}{4}}}{2^{\frac{1}{\alpha}(1+\frac{1}{2})}}  \left({ \int_{-\infty}^0 q(x)  \left(e^{4d|x|^\alpha} \int_{-\infty}^{(2^{1/\alpha})^2 x}   h^{4}\left(\frac{u}{(2^{1/\alpha})^2}\right)e^{-d|u|^\alpha} du\right)^2dx }\right)^{\frac{1}{4}}. \end{align*}
Iterating this procedure $m\in \mathbb{N}$ times we deduce 
$$I^-\le I_1^-\le\ldots \le  I_m^-$$
with $I_m^-$ given by 
\begin{align*} & \frac{\gamma_q^{N(m)-1} }{2^{\frac{1}{\alpha}N(m)}}  \left(\int_{-\infty}^0 q(x)  \left(e^{2^md|x|^\alpha} \int_{-\infty}^{(2^{1/\alpha})^{m} x}   h^{2^{m}}\left(\frac{u}{(2^{1/\alpha})^{m}}\right)e^{-d|u|^\alpha} du\right)^2dx\right)^{\frac{1}{2^m}},\end{align*}				
where  $N(m)=1+\frac{1}{2}+\frac{1}{4}+\ldots+\frac{1}{2^m}$. Bounding $h^{2^{m}}\left(\frac{u}{(2^{1/\alpha})^{m}}\right)$ by $(||h||_\infty)^{2^m}$ simplifies the above into
\begin{align*} & \frac{(||h||_\infty)^2 \gamma_q^{N(m)-1} }{2^{\frac{1}{\alpha}N(m)}}  \left(\int_{-\infty}^0 q(x)  \left(e^{2^md|x|^\alpha} \int_{-\infty}^{(2^{1/\alpha})^{m} x}   e^{-d|u|^\alpha} du\right)^2dx\right)^{\frac{1}{2^m}}.\end{align*}
Since the mapping $y\mapsto
\eta(y):=e^{d|y|^\alpha}\int_{-\infty}^ye^{-d|u|^\alpha}du$ attains
its maximal value at 0 for $\alpha\geq1$ (indeed, 
\begin{align*}\eta'(y)& =1-e^{d|y|^\alpha}d\,\alpha|y|^{\alpha-1}\int_{-\infty}^ye^{-d|u|^\alpha}du
  \\
&
\geq1-e^{d|y|^\alpha}\int_{-\infty}^yd\alpha|u|^{\alpha-1}e^{-d|u|^\alpha}du=0,\end{align*}
hence $\eta$ is monotone increasing), the interior of the parenthesis
becomes 
\begin{align*}
&\int_{-\infty}^0 q(x)  \left(e^{2^md|x|^\alpha}
  \int_{-\infty}^{(2^{1/\alpha})^{m} x}   e^{-d|u|^\alpha}
  du\right)^2dx \leq \int_{-\infty}^0 q(x) \frac{1}{c^2}dx  = \frac{\gamma_q}{c^2}.
\end{align*}
Note that here we have used, for any support $S$, $\int_{-\infty}^0ce^{-d|u|^\alpha}du\leq 1$. Elevated to the power $1/(2m)$, this factor tends to $1$ as
$m\rightarrow\infty$. Since we also have $\lim_{m\rightarrow\infty}N(m)=2$ we finally obtain
$$I^-\leq \lim_{m\rightarrow\infty}I^-_m\leq \frac{(||h||_\infty)^2 }{2^{\frac{2}{\alpha}}}{\rm P}_q(X<0).$$ 
Similar manipulations allow to bound $I^+$ by
$\frac{(||h||_\infty)^2}{2^{\frac{2}{\alpha}}}{\rm
  P}_q(X>0)$. Combining both bounds then allows us to conclude that

$$\sqrt{{\rm E}_q \left[ (f_h^p(X))^2 \right]}\leq\frac{||h||_\infty}{2^{\frac{1}{\alpha}}},$$ 
hence the claim holds. \end{proof}

This result of course holds true without worrying about
$f_h^p\in\mathcal{F}(q)$. However, in order to make use of these
bounds in the present context, the latter condition has to be taken
care of. For densities of the form \eqref{eq:15}, one easily sees that
$f_h^p\in\mathcal{F}(q)$ for all (differentiable and) bounded
densities $q$ for $\alpha>1$, with the additional assumption, for
$\alpha=1$, that $\lim_{x\rightarrow\pm\infty}q(x)=0$. 

\begin{example} Take $p=\phi$, the standard Gaussian. Then, from \eqref{eq:17},
\begin{equation}\label{eq:27}
  \kappa_{\mathcal{H}}^p \le \frac{1}{\sqrt2}\sup_{l\in \mathcal{H}}    \|l-{\rm E}_\phi \left[ l(X)
    \right]\|_{\infty}.
\end{equation}
Comparing with the bounds from Example
\ref{ex:bounding-constants} we see that \eqref{eq:27} significantly
improves on the constants in cases (i) and (iii); it is slightly worse
in case (ii).
  \end{example}

\section{Applications}
\label{sec:applications}

A wide variety of
probability distances can be written under the form 
\eqref{eq:hdizst}. For instance the
total variation distance  is given by
\begin{align*}
  d_{\mathrm{TV}}(p, q) = \sup_{A \subset \R} \left| \int_{A} (p(x) -
    q(x) )dx   \right| = \frac{1}{2} \sup_{h \in \mathcal{H}_{B[-1,1]}} \left| {\rm E}_p
    \left[ h(X) \right] - {\rm E}_q \left[ h(X) \right] \right|
\end{align*}
with $\mathcal{H}_{B[-1,1]}$ the class of
Borel functions in $[-1,1]$, 
the  Wasserstein distance is given by
$$
  d_{\mathrm{W}}(p, q) = \sup_{h \in \mathcal{H}_{{\rm Lip} 1}} \left| {\rm E}_p
    \left[ h(X) \right] - {\rm E}_q \left[ h(X) \right] \right|
$$
with $\mathcal{H}_{{\rm Lip} 1}$ the class of Lipschitz-1 functions on $\R$ and the
 Kolmogorov distance is given by
\begin{align*}
  d_{\mathrm{Kol}}(p, q)  = \sup_{z \in \R} \left| \int_{-\infty}^z (p(x) -
    q(x) )dx   \right| = \sup_{h \in \mathcal{H}_{HL}} \left| {\rm E}_p
    \left[ h(X) \right] - {\rm E}_q \left[ h(X) \right] \right|
\end{align*}
with $\mathcal{H}_{HL}$ the class of indicators of lower half lines. We
refer to \cite{GS02} for more examples and for an interesting overview
of the relationships between these probability metrics.

Specifying the class $\mathcal{H}$ in
Theorem~\ref{cor:gener-bound-dist} allows to bound all such
probability metrics in terms of the generalized Fisher information
distance \eqref{eq:6}. It remains to compute the constant
\eqref{eq:5}, which can be done for all $p$ of the form \eqref{eq:15}
 through \eqref{eq:17}.  The following result illustrates these computations in
several important  cases.



\begin{corollary}\label{distances}
Take  $p\in\mathcal{G}$ as in \eqref{eq:15} and $q\in\mathcal{G}$ such
that $S_q= S$. For $\alpha>1$, suppose that $q$ is (differentiable and) bounded over $S$; for $\alpha=1$, assume moreover that $q$ vanishes at the infinite endpoint(s) of $S$. Then we have the following inequalities:
\begin{enumerate}
\item
$$d_{\rm TV}(p,q)\leq
{2^{-\frac{1}{\alpha}}}\sqrt{\mathcal{J}(p,q)}$$
\item 
$$d_{\rm Kol}(p,q)\leq
{2^{-\frac{1}{\alpha}}}\sqrt{\mathcal{J}(p,q)}$$
\item $$d_{\rm W}(p,q)\leq
  \frac{\sup_{l\in\mathcal{H}_{{\rm Lip} 1}}||l-{\rm
      E}_p[l(X)]||_\infty}{2^{\frac{1}{\alpha}}}\sqrt{\mathcal{J}(p,q)}$$
\item$$d_{L^1}(p,q)=\int_S | p(x) - q(x) | dx \leq
2^{1-\frac{1}{\alpha}}\sqrt{\mathcal{J}(p,q)}.$$
\end{enumerate}

\noindent If, for all $y\in S$, $q$ is such that the
  function $f_l^p(x)=e^{d|x|^\alpha}(\mathbb{I}_{[y,b)}(x)-P(x))$,
  where $P$ denotes the cumulative distribution function associated
  with $p$, belongs to $\mathcal{F}(q)$, then
$$d_{\rm sup}(p,q)=\sup_{x\in \R} \left|p(x) - q(x)  \right|\leq \sqrt{\mathcal{J}(p,q)}.$$


\end{corollary}


\begin{proof}
The first three points follow immediately from the definition of the
distances and Theorems \ref{cor:gener-bound-dist} and
\ref{prop:bound}. To show the fourth, note that 
  \begin{align*}
  \int_S|p(x)-q(x)|dx & =  {\rm E}_p[l(X)]-{\rm E}_q[l(X)]
  \end{align*}
for $
  l(u)=\mathbb{I}_{[p(u)\geq q(u)]}-\mathbb{I}_{[q(u)\geq
  p(u)]} = 2 \mathbb{I}_{[p(u)\geq q(u)]}-1.$ For the last case note
that 
$$d_{\rm sup}(p,q):=\sup_{y\in S}|p(y)-q(y)|=\sup_{y\in S}|{\rm E}_p[l_y(X)-{\rm E}_q[l_y(X)]|$$ 
for $l_y(x)=\delta_{\{x=y\}}$ the Dirac delta function in $y\in
S$. The computation of the constant 
$\kappa_{\mathcal{H}}^p$ in this case requires a different approach
from our Theorem \ref{prop:bound}. We defer this to the Appendix. 
\end{proof}

We conclude this section, and the paper, with explicit computations in
the Gaussian case $p=\phi$, hence for the classical Fisher information distance.   From here on we adopt the more standard
notations and write $\mathcal{J}(X)$  instead of $\mathcal{J}(\phi, q)$, for $X$ a
random variable with density $q$ (which has support $\R$). Immediate
applications of the above yield
  \begin{equation*}
        \int_S \left| \phi(x) - q(x) \right|dx \le \sqrt2
        \sqrt{\mathcal{J}(X)},
  \end{equation*}
which is the second inequality in \cite[Lemma 1.6]{MR2128239} (obtained by
entirely different means). 
Similarly we readily deduce 
  \begin{equation*}
\sup_{x\in \R} \left| \phi(x) - q(x) \right| \le    \sqrt{\mathcal{J}(X)};
  \end{equation*}
this is  a significant improvement on the constant in   \cite{MR2128239,
  Sh75}.  

Next further suppose  that $X$ has density $q$ with  mean $\mu$ and variance
 $\sigma^2$. Take $Z \sim p$ with 
 $p=\phi_{\mu_0, \sigma_0^2}$, the Gaussian with mean $\mu_0$ and
variance $\sigma_0^2$. Then
\begin{align*}
  \mathcal{J}(X) & = {\rm E}_q \left[ \left(
      \frac{q'(X)}{q(X)}+ \frac{X-\mu_0}{\sigma_0^2}
    \right)^2 \right]  =  I(X) + \frac{(\mu-\mu_0)^2}{\sigma_0^4} + \frac{1}{\sigma_0^2}
\left( \frac{\sigma^2}{\sigma_0^2}-2 \right),  
\end{align*}
where $I(X) = {\rm E}_q \left[ (q'(X)/q(X))^2 \right]$ is the Fisher
information of the random variable  $X$. 
General bounds are thus also obtainable from \eqref{eq:4} in terms
of 
\begin{equation*}
\Psi :=\Psi(\mu, \mu_0, \sigma, \sigma_0) =    \frac{(\mu-\mu_0)^2}{\sigma_0^4} + \frac{1}{\sigma_0^2}
\left( \frac{\sigma^2}{\sigma_0^2}-1 \right).
\end{equation*}
and the quantity
\begin{equation*}
  \Gamma(X) =  I(X) - \frac{1}{\sigma_0^2},
\end{equation*}
referred to as the \emph{Cram\'er-Rao functional} for $q$ in 
\cite{MW90}. 
In particular,  we deduce from Theorem \ref{prop:bound} and the definition of the total
variation distance that 
$$
  d_{\rm TV}(\phi_{\mu_0, \sigma_0^2}, q) \le \frac{1}{\sqrt{2}} \sqrt{
      \Gamma(X) + \Psi}.
$$
This is an improvement (in the constant) on \cite[Lemma
 3.1]{MW90}, and is   also related to \cite[Corollary 1.1]{CPU94}. 
Similarly, taking $\mathcal{H}$
the collection
of indicators for lower half lines we can use \eqref{eq:10} and the bounds from
\cite[Lemma 2.2]{ChSh05} 
to deduce 
$$
d_{\rm Kol}(\phi_{\mu_0, \sigma_0^2},q) \le \frac{\sqrt{2\pi}}{4}\sigma_0 \sqrt{
      \Gamma(X) + \Psi}.  
$$
Further specifying  $q = \phi_{\mu_1, \sigma_1^2}$ we see that 
  \begin{equation*}
\sigma_0 \sqrt{
      \Gamma(X) + \Psi } \le  \frac{\left|
      \sigma_1^2- \sigma_0^2 \right|}{\sigma_0\sigma_1} + \frac{\left|
      \mu_1- \mu_0 \right|}{\sigma_0}, 
  \end{equation*}
to be compared with \cite[Proposition 3.6.1]{NP11}.  Lastly take $Z\sim \phi$ the standard Gaussian and $X\stackrel{d}{=}F(Z)$ for $F$ some
  monotone increasing function on $\R$ such that  $f = F'$ is defined everywhere. Then
  straightforward computations yield 
  \begin{equation*}
    I(X) = {\rm E} \left[ \left(\frac{ \psi_f(Z) + Z}{f(Z)} \right)^2
    \right], 
  \end{equation*}
with $\psi_f = (\log f)'$. In particular, if  $F$ is a random function
of the form $F(x) = Y x$ for $Y>0$
some random variable independent of $Z$, then simple conditioning
shows that the above becomes  
\begin{equation*}
  I(X) = {\rm E} \left[ \frac{Z^2}{Y^2} \right] = {\rm E} \left[ \frac{1}{Y^2} \right],
\end{equation*}
 so that 
\begin{equation*}
  d_{\rm TV}(\phi,q_X) \le \frac{1}{\sqrt2}  \sqrt{ {\rm E} \left[
      \frac{1}{Y^2} \right] - 1 +{\rm E}(Y^2 -1)}
\end{equation*}
where $q_X$ refers to the density of $X\stackrel{d}{=}YZ$. This last inequality is to be compared with \cite[Lemma 4.1]{CPU94} and also \cite{Sh87}.

\appendix

\section{Bounds for the supremum norm}
\label{sec:technical-proof}

  First note that, for $l_y(x)=\delta_{\{x=y\}}$, the solution
  $f_{l_y}^p(x)$ of the Stein equation \eqref{eq:steineq22} is of the
  form 
$$\frac{1}{p(x)}\int_{a}^x(\delta_{\{z=y\}}-p(y))p(z)dz=\frac{p(y)(\mathbb{I}_{[y,b)}(x)-P(x))}{p(x)}.$$
For all densities $q$ such that $f_{l_y}^p(x) \in \mathcal{F}(q)$, Theorem~\ref{cor:gener-bound-dist} applies and
yields
$${\rm sup}_{y\in S}|p(y)-q(y)|\leq {\rm sup}_{y\in S}p(y)\sqrt{{\rm
    E}_q[(\mathbb{I}_{[y,b)}(X)-P(X))^2/(p(X))^2]}\sqrt{\mathcal{J}(p,q)},$$
where $b$ is either $0$ or $+\infty$. We now prove that 
\begin{equation*}
{\rm sup}_{y\in S}p(y)\sqrt{{\rm
    E}_q[(\mathbb{I}_{[y,b)}(X)-P(X))^2/(p(X))^2]}\leq 1  
\end{equation*}
for $p(x)=c\,e^{-d|x|^\alpha}$ and any density $q$ satisfying the assumptions
of the claim. To this end note that straightforward manipulations lead to 
\begin{align*}
  {\rm E}_q[\left(\mathbb{I}_{[y, b)}(X) - P(X)\right)^2/(p(X))^2]&\\
  & \hspace{-4cm} = \frac{1}{c^2}\int_{a}^b q(x)e^{2d|x|^\alpha}(\mathbb{I}_{[y, b)}(x)-P(x))^2dx\\
  & \hspace{-4cm}= \frac{1}{c^2}\int_{a}^y q(x)e^{2d|x|^\alpha}(P(x))^2dx +\frac{1}{c^2}  \int_{y}^b q(x)e^{2d|x|^\alpha}(1-P(x))^2dx\\
  & \hspace{-4cm}\leq \frac{1}{c^2} e^{2d|y|^\alpha}(P(y))^2\int_{a}^y q(x)dx +\frac{1}{c^2} e^{2d|y|^\alpha}(1-P(y))^2\int_y^b q(x)dx\\
  & \hspace{-4cm}= \frac{1}{c^2} e^{2d|y|^\alpha}(P(y))^2 +
  \frac{1}{c^2}e^{2d|y|^\alpha}(1-2P(y)){\rm P}_q(X\geq y),
\end{align*}
where the inequality is due to the fact that $e^{2d|x|^\alpha}P(x)$
(resp., $e^{2d|x|^\alpha}(1-P(x)))$ is monotone increasing (resp.,
decreasing) on $(a,y)$ (resp., $(y,b)$); see the proof of
Theorem~\ref{prop:bound}. This again directly leads to
\begin{align*}
& {\rm E}_q[\left(\mathbb{I}_{[y, b)}(X) - P(X)\right)^2/(p(X))^2]\\
&\leq \sup_{y\in (a,b)}\left(ce^{-d|y|^\alpha}\sqrt{\frac{1}{c^2}e^{2d|y|^\alpha}((P(y))^2 +  (1-2P(y)){\rm P}_q(X\geq y)}\right)\\
&=\sup_{y\in(a,b)}\left(\sqrt{(P(y))^2 +  (1-2P(y)){\rm P}_q(X\geq y)}\right).
\end{align*}
This last expression is equal to 1.

\bibliographystyle{amsplain}
\bibliography{../../bibliography/biblio_ys_stein}

\end{document}